\documentclass[12pt]{article}

\usepackage{amsmath, amsthm, amscd, amsfonts, amssymb}

\newtheorem{theorem}{Theorem}[section]
\newtheorem{proposition}[theorem]{Proposition}
\newtheorem{claim}{Claim}

\newtheorem{lemma}[theorem]{Lemma}

\newtheorem{remark}[theorem]{Remark}

\def\PP{\mathbb{P}}
\def\edo{\end{document}}
\def\loc{_{\rm loc}}
\def\divv{{\rm div }}
\def\rrd{{\mathbb{R}^d}}

\def\calf{{\mathcal{F}}}
\def\calm{{\mathcal{M}}}

\def\cald{{\mathcal{D}}}
\def\calx{{\mathcal{X}}}

\def\call{{\mathcal{L}}}

\def\calp{{\mathcal{P}}}

\def\vsp{\vspace*{1,5mm}\\ }

\def\mk{\medskip }

\def\n{\noindent }
\def\dd{\displaystyle}

\def\barr{\begin{array}}
\def\earr{\end{array}}

\def\bit{\begin{itemize}}
\def\eit{\end{itemize}}

\def\FP{Fokker--Planck}

\def\1{^{-1}}

\def\E{{\mathbb{E}}}
\def\nn{{\mathbb{N}}}
\def\PP{{\mathbb{P}}}
\def\rr{{\mathbb{R}}}
\def\9{{\infty}}
\def\lbb{{\lambda}}
\def\wt{\widetilde}
\def\ov{\overline}
\def\vf{{\varphi}}

\def\ooo{{\Omega}}
\def\pp{{\partial}}
\def\vp{{\varepsilon}}

\def\pas{\mathbb{P}\mbox{-a.s.}}
\def\ff{\forall }
\def\({\left(}
\def\){\right)}
\def\<{\left<}
\def\>{\right>}

\def\fpe{Fokker--Planck equation}

\def\mkve{McKean--Vlasov equation}

\title{Probabilistic representation of solutions to~the~parabolic $p$-Laplace  equation} 
\author{Viorel Barbu\thanks{Octav Mayer Institute of Mathematics of  Romanian Academy,  Ia\c si, Romania.  Email:~vb41@uaic.ro} \and Michael Röckner\footnote{Faculty of Mathematics, Bielefeld University, 33615 Bielefeld, Germany and Academy of Mathematics and System Sciences, CAS, Beijing.}} 
\date{}

\begin{document}
\maketitle

\begin{abstract}
This work is concerned with the probabilistic representation of solutions to the $p$-Laplace evolution equation $\frac{\pp u}{\pp t}=\divv(|\nabla u|^{p-2}\nabla u)$ in $(0,\9)\times\rrd$, $u(0,x)=u_0(x),$ $x\in\rrd$. One proves that, if $p\ge4$, 
and if $u_0$ is a probability density with compact support and $u_0\in L^2$, $|\nabla u_0|\in L^\9$, then  
$u$ can be represented as $u(t,x)dx=\call_{X(t)}(dx)$, where $\call_{X(t)}$ denotes the time marginal law of $X$ at time $t$ with $X$ being a probabilistically weak solution to a corresponding  McKean--Vlasov stochastic diffe\-ren\-tial equation. This result is based on a new second order global re\-gu\-la\-rity result for the weak solutions to the parabolic $p$-Laplace equation.\\
{\bf MSC Codes:} 35B40, 35Q84, 60H10.\\
{\bf Keywords:} \fpe, stochastic, semigroup, \mkve, superposition principle.
\end{abstract} 

\section{Introduction}\label{s1}
Consider herein the nonlinear parabolic Cauchy problem
\begin{equation}\label{e1.1}
\barr{ll}
\dd\frac{\pp u}{\pp t}(t,x)=\divv(|\nabla u(t,x)|^{p-2}\nabla u(t,x)),&t>0,\ x\in\rrd,\vsp
u(0,x)=u_0(x),&x\in\rrd,\earr\end{equation}
where $d\ge1$ and  $1<p<\9$.\newpage

Following \cite{3} (see, also, \cite{4}), we associate with \eqref{e1.1}, rewritten as the \FP\ equation
\begin{equation}\label{e1.2}
\barr{ll}
\dd\frac{\pp u}{\pp t}-\Delta(|\nabla u|^{p-2}u)+\divv(\nabla(|\nabla u|^{p-2})u)=0,&t>0,\ x\in\rrd,\vsp
u(0,x)=u_0(x),&x\in\rrd,\earr\end{equation}
where $u_0$ is a probability density, the McKean-Vlasov SDE
\begin{equation}\label{e1.3}
\barr{rl}
dX(t)\!\!\!&=\nabla(|\nabla u(t,X(t))|^{p-2})dt+\sqrt{2}|\nabla u(t,X(t))|^{\frac{p-2}2}dW(t),\ t>0,\vsp\call_{X(t)}\!\!\!&=u(t,x)dx,\ t\ge0,\earr
\end{equation}
where $W(t)$ is a (standard) $d$-dimensional Brownian motion on some probability space $(\ooo,\calf,\mathbb{P})$.

We recall that $u\in L^1_{\rm loc}(0,\9;W^{1,1}_{\rm loc}(\rrd))$ is a {\it distributional solution} to \eqref{e1.2}~if
\begin{equation}\label{e1.4}
\barr{c}
|\nabla u|^{p-2}\in L^1\loc((0,\9);W^{1,1}\loc(\rrd)), \ 
|\nabla u|^{p-2}u\in L^1\loc(0,\9;\rrd),\vsp
\nabla(|\nabla u|^{p-2})u\in L^1\loc((0,\9)\times\rrd;\rrd),\earr \end{equation}
and
\begin{equation}\label{e1.5}
\barr{l}
\dd\int^\9_0\int_\rrd u\(\frac{\pp\vf}{\pp t}+|\nabla u|^{p-2}\Delta\vf+\nabla(|\nabla u|^{p-2})\cdot\nabla\vf\)dxdt\vsp
\qquad+\dd\int_\rrd\vf(0,x)u_0(x)dx=0,
\earr\end{equation}for all $\vf\in C^\9_0([0,\9)\times\rrd).$

A {\it probabilistically weak solution} to \eqref{e1.3} is a triple consisting of a filtered probability space $(\ooo,\calf,(\calf_t)_{t\ge0},\PP)$, a continuous  $(\calf_t)$-adapted $\rrd$-valued process $X=(X(t))_{t\ge0}$ and an $(\calf_t)$-Brownian motion $W=(W(t))_{t\ge0}$ such that
\begin{equation}\label{e1.5prim}
\E\left[\int^T_0(|\nabla(|\nabla u(t,X(t))|^{p-2})+|\nabla u(t,X(t))|^{p-2})dt\right]<\9,\ \ff T>0,\end{equation}
 and $\pas$
$$\barr{r}
X(t)=X_0+\dd\int^t_0\nabla(|\nabla u(s,X(s))|^{p-2})ds+\sqrt{2}\dd\int^t_0|\nabla u(s,X(s))|^{\frac{p-2}2}dW(t),\\ \ff t\ge0.\earr$$
By \cite[Section 2]{1a} (see also \cite[Chapter 5]{2}),   for any $\calp$ (= all probability measures on $\rrd$)-valued weakly continuous distributional solution $u$ to \eqref{e1.2} satisfying the following global regularity result\newpage \ \vspace*{-16mm}
\begin{equation}\label{e1.5secund}
\int^T_0\int_\rrd(|\nabla(|\nabla u(t,x)|^{p-2}
+|\nabla u(t,x)|^{p-2})u(t,x))dxdt<\9,\ \ff t\ge0,	\end{equation}
(which is also a necessary requirement in view or \eqref{e1.5prim}),
 there exists a (probabilistically) weak solution $X=(X(t))_{t\ge0}$ to SDE \eqref{e1.3} such that $\call_{X(t)}(dx)=u(t,x)dx$, $t\ge0$. This means that we obtain a probabilistic representation of a solution $u$ to the parabolic $p$-Laplace equation, rewritten as \eqref{e1.2}, in the sense that $u$ is a time marginal law density of the solution $X$ to SDE \eqref{e1.3}. In \cite{3}, such a result was proved for the Barenblatt solutions to equation \eqref{e1.1} and extended to Leibenson's equation in \cite{4}. In these cases, the cor\-res\-po\-nding solution $X$ to \eqref{e1.3} is even a probabilistically strong solution and the path laws of the solutions to \eqref{e1.3} form a nonlinear Markov process in the sense of McKean. The extension of these probabilistic representation results to a more general class of distributional solutions to equation \eqref{e1.1} is a challenging objective and this work is a step of this program. More precisely, we prove the global regularity result in Theorem \ref{t3.1} below, which is in fact stronger than \eqref{e1.5secund}, for bounded initial probability densities. Then, as just mentioned, by applying \cite[Section 2]{1a} which, in turn, is based on the (linear) {\it superposition principle} from \cite[Theorem 2.5]{7}, we obtain a weak solution to \eqref{e1.3}. This is formulated and proved in Section \ref{s4} (see Theorem \ref{t4.1}) for $p\ge4$ if $u(0)=u_0$ has compact support $u_0\in L^2(\rrd)$, $|\nabla u_0|\in L^\9(\rrd)$, and is a pro\-ba\-bi\-li\-ty density. For this purpose, in Sections \ref{s2} and  \ref{s3} we prove some sharp global second order estimates for weak solutions to equation \eqref{e1.1} which to the best of our knowledge are new. We refer at this point to a related local estimate in the paper \cite{9b}. However, it is in time only for  compact intervals in $(0,T)$ which is not sufficient for our purpose (see Remark \ref{r3.3}). 
 
 Finally, we would like to emphasize that probabilistic representations of solutions to nonlinear parabolic equations of the type as in this paper was first suggested in McKean's seminal paper \cite{9bb}. So, we realize his vision in this paper for the parabolic $p$-Laplace equation for a large class of initial data. 
 
  As already anticipated by McKean, having a solution $X=(X(t))_{t\ge0}$ to the SDE \eqref{e1.3} bears valuable information for its time marginal law $u$, i.e., the solution to \eqref{e1.2}, since one has all methods from stochastic analysis at hand to analyze $X$,  as, e.g., applying It\^o's formula, (semi) martingale theory, stochastic solutions to Dirichlet problems and, in particular, Malliavin calculus to prove further regularity of $u$. We refer to the recent work \cite{5c}, where some parts of such methods have already been exploited for the generalized porous media equation. To do the same for the parabolic $p$-Laplace equation \eqref{e1.2} is our subject in a forthcoming work.\mk

\noindent{\bf Notations.} For $1\le p\le\9$, $L^p(\rrd)$ (also denoted $L^p$) stands for the usual real-valued $L^p$-spaces on $\rrd$ with the norm $|\cdot|_p$. For an open set $B\subset\rrd$, $W^{1,p}(B)$  is the Sobolev space $\{u\in L^p(B);$ $D_iu\in L^p(B),\ i=1,...,d\}$ with the standard norm $\|\cdot\|_{W^{1,p}(B)}$,  where $D_i u=\frac\pp{\pp x_i}u$ is the  distributional derivative of $u$.  We denote by $W^{-1,p'}(B)$, $\frac1{p'}=1-\frac1p,$ the dual space of $W^{1,p}(B)$. We set $W^{1,p}=W^{1,p}(\rrd)$, denote by $W^{1,p}\loc$ the corresponding local space and by $W^{2,p}$ the space $\{u\in L^p;\ D_ju\in L^p,\ D_iD_ju\in L^p;\ i,j=1,...,d\}$. By $C^\9_0((0,\9)\times\rrd)$ we denote the space of infinitely differentiable real valued functions with compact support in $(0,T)\times\rrd$. By $\cald'(\rrd)$, respectively $\cald'((0,T)\times\rrd)$ we denote the space of Schwartz distributions on $\rrd$ and $(0,T)\times\rrd$, respectively. We use also the notations
$$W^{1,2}(B)=H^1(B),\ W^{1,2}\loc=H^1\loc,\ W^{2,2}=H^2,\ W^{-1,2}=H\1.$$
Given a Banach space $H$ we denote by $L^p(0,T;H)$, $0<T\le\9$, the space of Bochner measurable $p$-integrable functions $u:(0,T)\to H$. By $C([0,T];H)$ we denote the space of $H$-valued continuous functions $u:[0,T]\to H$. By $W^{1,p}([0,T];H)$ we denote the Sobolev space 
$$\left\{u\in L^p(0,T;H);\frac{du}{dt}\in L^p(0,T;H)\right\},$$ where $\frac{du}{dt}$ is the $H$-valued distributional derivative of $u$ (see, e.g., \cite{1},\ p.23). We recall that each $u\in W^{1,p}([0,T];H)$ is absolutely continuous on $[0,T]$ and $\frac{du}{dt}(t)$ exists, a.e. on $(0,T)$. Denote by $\cald'(0,T;H\1)$ the space of $H\1$-valued distributions on $(0,T)$. By $\calm_b$ we denote the space of bounded-Radon measures on $\rrd$. For each $0<T\le\9$, $L^p((0,T)\times\rrd)$ is the standard space of Lebesgue $p$-integrable real valued functions on $(0,T)\times\rrd.$ The scalar product of $L^2$ is denoted by $(\cdot,\cdot)_2$.  

\section{General existence theory for the Cauchy problem \eqref{e1.1}}\label{s2}
\setcounter{equation}{0}

Herein we shall briefly recall some standard existence results for the weak solutions to equation \eqref{e1.1}. 
Problem \eqref{e1.1} can be treated as an infinite dimensional Cauchy problem in the space $H=L^2$, namely,
\newpage 
\begin{equation}\label{e2.1}
\barr{l}
\dd\frac{du}{dt}(t)+Au(t)=0,\ t>0,\vsp
u(0)=u_0,\earr\end{equation}
where the operator $A:D(A)\subset H\to H$ is defined by
\begin{equation}\label{e2.2}
\barr{rl}
Au\!\!\!&=-\divv(|\nabla u|^{p-2}\nabla u),\ \ff u\in D(A),\vsp
D(A)\!\!\!&=\{u\in L^2;\nabla u\in L^p,\divv(|\nabla u|^{p-2}\nabla u)\in L^2\}.\earr\end{equation}
Here, $\nabla$ and $\divv$ are taken in the sense of Schwartz distributions on $\rrd$ and $\frac{du}{dt}$ is the distributional 
 derivative of the function $u:[0,\9)\to H$. More exactly, we look for solutions $u$ in the Sobolev space
$$W^{1,2}\loc([0,\9);H)=\{u\in W^{1,2}([0,T];H),\ \ff T>0\},$$
which satisfy \eqref{e1.2} in the weak or variational sense. To get the main existence result we shall prove first 
 the proposition below. 

\begin{proposition}\label{p2.1}
The operator $A$ is $m$-accretive $($maximal monotone$)$ in $H\times H$ and
\begin{equation}\label{e2.3}
A=\pp\Phi,\end{equation}
where $\Phi:H\to\,]-\9,+\9]$ is the function
\begin{equation}\label{e2.4}
\Phi(u)=\left\{\barr{ll}
\dd\frac1p\int_\rrd|\nabla u(x)|^pdx&\mbox{ if }|\nabla u|\in L^p,\vsp
+\9&\mbox{ otherwise},\earr\right.\end{equation}
which is convex and lower semicontinuous. 

Herein, $\pp\Phi:H\to H$ is the subdifferential of the function $\Phi$, that is,
\begin{equation}\label{e2.5}
\pp\Phi(u)=\{w\in H;\,\Phi(u)\le\Phi(v)+(w,u-v)_2,\ \ff v\in H\}.\end{equation}
\end{proposition}
This function is called the {\it potential} of the operator $A$.

\begin{proof} By \eqref{e2.2} we see that
\begin{equation}\label{e2.6}
(Au,u-v)_2=\int_\rrd|\nabla u|^{p-2}\nabla u\cdot(\nabla u-\nabla v)dx,\ \ff u,v\in D(A),\end{equation}
and, therefore, $A$ is monotone (accretive), that is,
\begin{equation}\label{e2.7}
(Au-Av,u-v)_2\ge0,\ \ff u,v\in D(A).\end{equation}We also have, by
\eqref{e2.4}--\eqref{e2.6},
$$(Au,u-v)\ge\Phi(u)-\Phi(v),\ \ff u\in D(A),\ v\in D(\Phi)=\{v\in H;|\nabla v|\in L^p\}.$$Hence,
\begin{equation}\label{e2.8}
Au\in\pp\Phi(u),\ \ff u\in D(A).\end{equation}
To complete the proof, it remains to be shown that $R(I+A)=H$. To this end, we fix $f\in H$ and consider the equation
\begin{equation}\label{e2.9}
u+Au=f.\end{equation}
Let $V=\{u\in L^2;\nabla u\in L^p\}$ with the norm
$$\|u\|_V=|u|_2+|\nabla u|_p$$and the dual $V'$. Clearly, we have
$$V\subset H\subset V'$$
algebraically and topologically.

Next, we consider the operator $A_V:V\to V'$ defined by
$${}_{V'}\<A_V(u),v\>_V=\int_\rrd|\nabla u|^{p-2}\nabla u\cdot\nabla v\,dx,\ \ff u,v\in V,$$
and note that it is monotone, that is,
$${}_{V'}\<A_V(u)-A_V(\bar u),u-\bar u\>_V\ge0,\ \ff u,\bar v\in V,$$
and demicontinuous, that is strongly-weakly continuous. Then, by the Minty--Browder theorem (see, e.g., \cite{1}, p.~36), the operator $A_V$ is maximal monotone in $V\times V'$. Moreover, the operator $\wt A_V(u)=u+A_V(u)$ is coercive in $V\times V'$, because

$${}_{V'}(\wt A_V(u),u)_V=|u|^2_2+|\nabla u|^p_p,\ \ff u\in V.$$
Since   $\wt A_V$ is also maximal monotone in $V\times V'$, it follows that the range $R(I+\wt A_V)$ of $I+\wt A_V$ is all of $V'$ and, in particular,   $H\subset R(I+\wt A_V)$, where $I$ is the identity operator in $H$. We have
$$Au=A_V(u),\ \ff u\in D(A)=\{u\in V;\,A_V(u)\in H\}.$$
Hence, $R(I+A)=H$ and so \eqref{e2.9} has a solution $u\in D(A)$, as claimed.

Since $\Phi:H\to\,]-\9,+\9]$ is convex and lower-semicontinuous, its subdifferential $\pp\Phi:H\to H$ is maximal monotone. As shown above, $A$ is maximal monotone and so, by \eqref{e2.8}, it follows that $A=\pp\Phi$, as claimed.\end{proof}

By Proposition \ref{p2.1} and the general existence theory for the Cauchy pro\-blem in the Hilbert space $H$ associated with subgradient operators $A=\pp\Phi$, we have (see, e.g., \cite{1}, pp.~143 and 158),
\begin{proposition}
	\label{p2.2} Let $u_0\in\ov{D(A)}$ $($the closure of $D(A)$ in $H)$. Then, for each $T>0$ there is a unique function $$u\in C([0,T];H)\cap\bigcap_{0<\delta<T} W^{1,2}([\delta,T];H)$$ 
	such that $u(t)\in D(A)$, a.e. $t\in(0,T)$, and
\begin{eqnarray}
&t^{\frac12}\,\dd\frac{du}{dt},t^{\frac12}Au\in L^2(0,T;H);\ \Phi(u)\in L^1(0,T)\label{e2.10}\\[1mm]
&\dd\frac{du(t)}{dt}+Au(t)=0,\mbox{ a.e. }t\in(0,T);\ u(0)=u_0.\label{e2.11}
\end{eqnarray}
If $\Phi(u_0)<\9$, then
\begin{eqnarray}
&\dd\frac{du}{dt}\in L^2(0,T;H),\label{e2.12}\\[1mm]
&\Phi(u(t))\le\Phi(u_0),\ \ff t\in[0,T].\label{e2.13}
\end{eqnarray}
Finally, if $u_0\in D(A)$, then
\begin{eqnarray}
&\dd\frac{du}{dt}\in L^\9(0,T;H),\ Au\in L^\9(0,T;H),\label{e2.14}\\[1mm]
&\dd\frac{d^+}{dt}u(t)+Au(t)=0,\ \ff t\in[0,T].\label{e2.15}
\end{eqnarray}
\end{proposition}

Applying Proposition \ref{p2.1} to the operator $A$ defined by \eqref{e2.2} on the space $H=L^2$, we get the following existence and uniqueness result for pro\-blem~\eqref{e1.1}.

\begin{theorem}\label{t2.3} Let $u_0\in L^2$. Then, for each $T>0$, there is a unique function
\begin{equation}\label{e2.16}
u\in C([0,T];L^2)\cap \cap\bigcap_{0<\delta<T} W^{1,2}([\delta,T];H)\end{equation}
such that $u(t)\in D(A)$, a.e. $t\in(0,T)$, and 

\begin{eqnarray}
&t^{\frac12}\,\dd\frac{du}{dt} \in L^2(0,T;L^2),\ t^{\frac12}Au\in L^2(0,T;L^2),\ |\nabla u|\in L^p(0,T;L^p),\label{e2.17}\\[1mm]
&\barr{l}
\dd\frac{du}{dt}(t)=\divv(|\nabla u(t)|^{p-2}\nabla u(t)),\mbox{ a.e. }t\in(0,T),\earr\label{e2.18}\\[1mm]
&\dd\frac12|u(t)|^2_2+\int^t_0\int_\rrd|\nabla u(s,x)|^pdxds=\frac12|u_0|^2_2,\ \ff t\ge0.\label{e2.18a}
\end{eqnarray}
If $|\nabla u_0|\in L^p$, then
\begin{eqnarray}
&\dd\frac{du}{dt}\in L^2(0,T;L^2),\ \divv(|\nabla u|^{p-2}\nabla u)\in L^2((0,T)\times\rrd),\label{e2.19}\\[1mm]&|\nabla u(t)|^p_p\le|\nabla u_0|^p_p,\mbox{ a.e. } t\in[0,T]. \label{e2.20}
\end{eqnarray}
Finally, if $u_0\in L^2,\ |\nabla u_0|\in L^p,\ \divv(|\nabla u_0|^{p-2}\nabla u_0)\in L^2,$ then
\begin{eqnarray}
&\dd\frac{du}{dt}\in L^\9(0,T;L^2),\ \divv(|\nabla u|^{p-2}\nabla u)\in L^\9(0,T;L^2),\label{e2.21}\\[1mm]
&\dd\frac{d^+}{dt}\,u(t)=\divv(|\nabla u(t)|^{p-2}\nabla u(t)),\ \ff t\in[0,T).\label{e2.22b}\end{eqnarray}
$($Herein, $\frac{du}{dt}(t)$ is the strong derivative of the function $u:[0,T]\to H$.$)$
\end{theorem}

In the following, we shall call such a function $u$ the {\it weak solution} to the parabolic $p$-Laplace equation \eqref{e1.1}.

We also have
\begin{theorem}\label{t2.4} If $u_0\in L^\9\cap L^2$, then $u\in L^\9((0,T)\times\rrd)$, $|u(t)|_\9\le|u_0|_\9,$ a.e. $t\in(0,T)$. 
	Moreover, if $u_0\in L^1\cap L^2$, then $u(t)\in L^\9(0,T;L^1)$   and, if $u_0\ge0$, a.e. on $\rrd$, then 
\begin{equation}
u\ge0, \mbox{ a.e. on $(0,T)\times\rrd$},\label{e2.24a}\end{equation}
and, if $p\ge d$, we have
\begin{equation}
	\dd	\int_\rrd u(t,x)dx=\int_\rrd u_0(x)dx,\ \ff t\in[0,T].\label{e2.22}
\end{equation}
\end{theorem}

\begin{proof} We  first prove that, if $u_0\in L^1\cap L^2$, then $u\in L^\9(0,T;L^1)$. Let $\calx$ be the function 
	
\begin{equation}\label{e2.22a}
\calx_\delta(r)=\left\{\barr{rl}
1&\mbox{ for }r\ge\delta,\vsp\dd\frac r\delta&\mbox{ for }|r|<\delta,\vsp-1&\mbox{ for }r\le-\delta.
\earr\right.
\end{equation}
Taking into account that $\calx_\delta(u(t))\in L^2$, $\ff t\in(0,T)$, we have (see, e.g., \cite{1}, p.~158, Lemma 4.1)
$$\(\frac d{dt}u(t),  \calx_\delta (u(t))\)_2=\frac d{dt}\int_\rrd  j_\delta(u(t,x))dx,\mbox{ a.e. }t\in(0,T),$$
where $j_\delta(v)=\int^v_0\calx_\delta(s)ds,$  $\ff v\in\rr$. We get by \eqref{e2.18} 
$$\dd\frac d{dt}\int_\rrd j_\delta(u(t,x))dx
=-\dd\int_\rrd|\nabla u(t,x)|^p\calx'_\delta(u(t,x))dx\le0,\mbox{ a.e. }t>0.$$
For $\delta\to0$, this yields, $\ff t\in[0,T]$,
\begin{equation}\label{e2.26a}
\dd\lim_{\delta\to0}
\int_\rrd j_\delta(u(t,x))dx
\le \dd\int_\rrd |u_0(x)|dx,
\end{equation}
and so   we get   
$$|u(t)|_1\le|u_0|_1,\ \mbox{ a.e. }\ff t\in[0,T],$$as claimed. 
Next, if $u_0\ge0$, a.e. on $\rrd$, we get in a similar way that
$$\frac d{dt}\int_\rrd u^-(t,x)dx\le0,\mbox{ a.e. }t\in(0,T),$$and, therefore, $u\ge0$, a.e. on $(0,T)\times\rrd$. 

Assume now that $u_0\in L^\9$. Then, by \eqref{e2.18} we see that
$$\barr{r}
\dd\frac d{dt}(u(t,x)-|u_0|_\9)-\divv(|\nabla(u-|u_0|_\9)|^{p-2}\nabla(u-|u_0|_\9))=0,\vsp
\hfill\mbox{ a.e. on }(0,T)\times\rrd,\earr$$
and multiplying by $(u-|u_0|_\9)^+$, we get after integration over  $(0,t)\times\rrd$
$$(u(t,x)-|u_0|_\9)^+=0,\mbox{ a.e. }(t,x)\in(0,T)\times\rrd.$$
Hence, $u\le|u_0|_\9$, a.e. on $(0,T)\times\rrd$ and in a similar way it follows $u\ge-|u_0|_\9$, a.e. on $(0,T)\times\rrd.$ Hence, $|u(t,x)|\le|u_0|_\9$, a.e. $(t,x)\in(0,T)\times\rrd$, as claimed.

To get \eqref{e2.22}, we note that by \eqref{e2.18}
$$\frac d{dt}(u(t),\vf_n)=-\int_\rrd|\nabla u(t,x)|^{p-2}
(\nabla u(t,x)\cdot\nabla\vf_n(x))dx,\mbox{ a.e. }t>0,$$
where $\vf_n(x)=\eta\(\frac{|x|^2}n\),$ $\eta\in C^2([0,\9)),$ $\eta(r)=1$, $\ff r\in[0,1]$, $\eta(r)=0$, $\ff r\in[2,\9)$. 
This yields
\begin{equation}\label{e2.27a}
\barr{ll}
\dd\int_\rrd \vf_n(x)u(t,x)dx=\!\!\!&
-\dd\int^t_0\!\int_\rrd|\nabla u(s,x)|^{p-2}
\nabla u(s,x){\cdot}\nabla\vf_n(x)dxds\vsp 
&\dd+\int_\rrd\!\!\vf_n(x)u_0(x)dx.\earr
\end{equation}
Taking into account that, for $B_n=\{x;\sqrt{n}\le|x|\le\sqrt{2n}\},$
$$\barr{l}
\left|\dd\int^t_0\int_\rrd|\nabla u(s,x)|^{p-2}\nabla u(s,x)\cdot\nabla\vf_n(x)dxds\right|\vsp
\qquad\le|\eta'|_\9\dd\frac{2\sqrt{2}}{\sqrt{n}}\int^t_0ds
\(\int_{[|x|\ge\sqrt{n}]}
|\nabla u(s,x)|^pdx\)^{\frac{p-1}p}
(V\!ol(B_n))^{\frac1p}\vsp 
\qquad\le C n^{\frac12\(\frac dp-1\)}\dd\int^t_0ds
\(\int_{[|x|\ge\sqrt{n}]}|\nabla u(s,x)|^pdx\)^{\frac{p-1}p}
\earr$$
and so, letting $n\to\9$ in \eqref{e2.27a}, we get \eqref{e2.22}. 
\end{proof}

An important property of equation \eqref{e1.1} is the finite speed of propagation of the solution $s$. Namely, we have (see, e.g., Theorem 3.4 in \cite{5prim}, and also~\cite{6})

\begin{proposition}\label{p2.5} Let $p>2$, $u_0\in L^2$, and let $u$ be the solution given by Theorem {\rm\ref{t2.3}}. Assume that
\begin{equation}\label{e2.23}
support\ u_0\subset B_R=\{x\in\rrd;\ |x|\le R\}.
\end{equation}
Then,
\begin{equation}\label{e2.24}
support\ u(t,\cdot)\subset B_{2R+R(t)},\ \ff t\ge0,\end{equation}
where
\begin{equation}\label{e2.25}
R(t)=C\,t^{\frac1{d(p-2)+p}}\,|u_0|^{\frac{p-2}{d(p-2)+p}}_1,\ t\ge0,\end{equation}
and $C>0$ is independent of $t$ and $u_0$.
\end{proposition}

\begin{remark}\label{r2.6a}\rm Under the hypotheses of Proposition \ref{p2.5} it follows by \eqref{e2.27a} that the condition $d\le p$ for the proof of \eqref{e2.22} is no longer necessary. 
\end{remark}

\begin{remark}\label{r2.6}\rm By Proposition \ref{p2.1} it follows also that the operator $A$ gene\-rates a continuous semigroup of contractions $S(t)=e^{-tA}$ in $H=L^2$ given~by
\begin{equation}\label{e2.26}
S(t)u_0=\lim_{n\to\9}\(I+\frac tn\,A\)^{-n}u_0\ \mbox{ strongly in }H,\ \ff t\ge0.\end{equation}
For $u_0\in L^2$ we have, therefore,
\begin{equation}\label{e2.27}
S(t)u_0=u(t),\ t\ge0,\ x\in\rrd,\end{equation}
where $u$ is the solution $u$ to \eqref{e1.1} given by Theorem \ref{t2.3}. By \eqref{e2.17}--\eqref{e2.21} it follows that $S(t)$ has a {\it smoothing effect} on initial data. Moreover, as seen in Theorem \ref{t2.4}, $S(t)$ leaves invariant the space $L^1\cap L^2$.
\end{remark}

\section{Second order estimates for solutions to \eqref{e2.18}}
\label{s3}
\setcounter{equation}{0}

Everywhere in the following $u$ is the weak solution to equation \eqref{e1.1} (equi\-va\-lently, \eqref{e2.11}), given by Theorem \ref{t2.3}.
\begin{theorem}	\label{t3.1}
	Let $p\ge4$, and let $u_0\in L^2$ with compact support in $\rrd$ and $|\nabla u_0|\in L^\9$. Assume that \eqref{e2.23} holds. 
Then, we have $\nabla u\in L^\9(0,T;L^\9)$, $\ff T>0$,  and
\begin{eqnarray}
&\dd\int^T_0\int_\rrd|\nabla u(t,x)|^{p-2}dxdt\le (T(\mu(T,R)))^{\frac2p}|u_0|^{\frac{2(p-2)}p}_2,\label{e3.2}\\[2mm]
&\barr{l}
\dd\int^T_0 \int_\rrd |\nabla(|\nabla u(t,x)|^{p-2})|dxdt\\
\qquad\quad
\le\dd\frac{\sqrt{d}\,(p-2)}2(2T\mu(T,R))^{\frac2p}|u_0|^{\frac{p-4}p}_2|\nabla u_0|_2,
\earr\label{e3.2a}
\end{eqnarray}
where $\mu(T,R)=V\!ol(B_{2R+R(T)})$ and $R(t)$ is given by \eqref{e2.25}. 
\end{theorem}

\n{\it Proof.} By \eqref{e2.18a} and \eqref{e2.24} we have
$$\barr{ll}
\dd\int^T_0\!\!\!\int_\rrd\!\!|\nabla u(t,x)|^{p-2}dxdt
\le\!\!\dd\int^T_0\!\!dt\(\int_\rrd|\nabla u(t,x)|^pdx\)^{\frac{p-2}p}\!\!(V\!ol(B_{2R+R(T)}))^{\frac2p}\vsp 
\le T^{\frac2p}(\mu(T,R))^{\frac2p}
\(\dd\int^T_0\int_\rrd|\nabla u(t,x)|^pdxdt\)^{\frac{p-2}p}
\le T^{\frac2p}(\mu(T,R))^{\frac2p}|u_0|^{\frac{2(p-2)}p}_2\earr$$
and so \eqref{e3.2} follows.\newpage 

 \n{\it Proof of \eqref{e3.2a}.} We note first that, since $u_0\in L^2$ and $|\nabla u_0|\in L^\9$, then by the Sobolev--Gagliardo--Nirenberg theorem combined with the Morrey theorem (see, e.g., \cite{8b}, pp.~278, 282) it follows that $u_0\in L^\9$. We approximate \eqref{e2.18}~by	
\begin{equation}\label{e3.3}
\barr{l}
\dd\frac{du}{dt}=\vp\Delta u+\divv(|\nabla u|^{p-2}\nabla u),\mbox{ a.e. } t\ge0,\ x\in\rrd,\vsp
u(0)=u_0,
\earr\end{equation}
where $\vp>0$. We may  rewrite \eqref{e3.3} as
\begin{equation}\label{e3.4}
\frac{du}{dt}+A_\vp u=0,\mbox{ a.e. } t\in(0,\9);\ u(0)=u_0,\end{equation}
where $A_\vp:D(A)\subset L^2\to L^2$ is given by
$$\barr{rl}
A_\vp u\!\!\!&=-\vp\Delta u-\divv(|\nabla u|^{p-2}\nabla u),\ \ff u\in D(A_\vp),\vsp
D(A_\vp)\!\!\!&=\{u\in H^{1};\ \nabla u\in L^p,\ \vp\Delta u+\divv(|\nabla u|^{p-2}\nabla u)\in L^2\}.\earr$$
Arguing as in the proof of Proposition \ref{p2.1}, it follows that $A_\vp$ is maximal monotone in $L^2\times L^2$ and
$A_\vp=\pp\Phi_\vp$, where $\Phi_\vp:L^2\to\,]-\9,+\9]$ is the potential function of the operator given by
$$\Phi_\vp(u)=\left\{\barr{cl}
\dd\int_\rrd\(\frac\vp2|\nabla u|^2+\frac1p|\nabla u|^p\)dx&\mbox{ if }\nabla u\in L^2\cap L^p,\vsp
+\9&\mbox{ otherwise}.\earr\right.$$
Then, by Proposition \ref{p2.2} it follows that \eqref{e3.3} has a unique solution $u_\vp\in C([0,\9);L^2)\cap L^\9(0,\9;W^{1,2})$ satisfying (see \eqref{e2.19}--\eqref{e2.20})
\begin{eqnarray}
&\dd\frac\vp2|\nabla u_\vp(t)|^2_2+\frac1p|\nabla u_\vp(t)|^p_p\le\frac\vp2|\nabla u_0|^2_2+\frac1p|\nabla u_0|^p_p,\ \ff t\ge0,\label{e3.5}\\[1mm]
&|u_\vp(t)|_2\le|u_0|_2,\ 
 \ff t\ge0,\label{e6}\\[1mm]
&\dd\frac{du_\vp}{dt}\in L^2(0,T;L^2),\ T>0,\label{e3.7}\\[1mm]
&\vp\Delta u_\vp+\divv(|\nabla u_\vp|^{p-2}\nabla u_\vp)\in L^2(0,T;L^2),\ \ff T>0.\label{e3.8}
\end{eqnarray}

\begin{claim}\label{c1} We have $u_\vp(t)\in H^2,\mbox{ a.e. }t\in(0,\9)$, and
\begin{equation}\label{e3.9}
\vp|\Delta u_\vp(t)|_2\le|f(t)|_2,\mbox{ a.e. }t\in(0,\9),\end{equation}
where $f(t)=\vp\Delta u_\vp(t)+\divv(|\nabla u_\vp(t)|^{p-2}\nabla u_\vp(t))$.\end{claim}

\begin{proof} We set
$$B_\lbb(u)=-\Delta(I-\lbb\Delta)\1 u=\frac1\lbb(u-(I-\lbb\Delta)\1 u),\ \ff u\in L^2.$$
We have
$$\barr{l}
\dd\int_\rrd|\nabla u_\vp|^{p-2}\nabla u_\vp\cdot\nabla(B_\lbb(u_\vp))dx\vsp
 \qquad=\dd\frac1\lbb\int_\rrd|\nabla u_\vp|^{p-2}\nabla u_\vp\cdot(\nabla u_\vp-\nabla(I-\lbb\Delta)\1u_\vp)dx\ge0,\earr$$
because $$|\nabla(I-\lbb\Delta)\1u_\vp|_p\le|
\nabla  u_\vp|_p,\ \ff\lbb>0.$$
This  yields
$$\vp\<\Delta u_\vp,\Delta(I-\lbb\Delta)\1u_\vp\>
\le-
(B_\lbb(u_\vp),f)_2\le|B_\lbb(u_\vp)|_2|f|_2.$$
Since
$$\barr{ll}
{}_{H^{-1}}\<\Delta u_\vp,\Delta(I-\lbb\Delta)\1u_\vp\>_{H^{1}}\!\!\!
&={\, }_{H^{-1}}\!\<\Delta u_\vp,(I-\lbb\Delta)\1\Delta u_\vp\>_{H^{1}}\vsp&\ge|\Delta(I-\lbb\Delta)\1 u_\vp|^2_2,\earr$$
we have
$$\vp|\Delta(I-\lbb\Delta)\1u_\vp|_2\le|f|_2,\ \ff\lbb>0,$$and so, letting $\lbb\to0$, we get \eqref{e3.9}, as claimed.\end{proof} 

\begin{claim}\label{c2}	For $\vp\to0$, we have, for all $T>0$,
\begin{equation}\label{e3.10}
u_\vp(t)\to u(t)\mbox{ in }L^2\mbox{ uniformly on }[0,T],\end{equation}
where $u$ is the solution given by Theorem {\rm\ref{t2.3}}.
\end{claim}

The latter follows by the Kato--Trotter theorem for nonlinear semigroups of contractions (see, e.g., \cite{1}, pp.~168, 169) because we have

\begin{lemma}\label{l3.1} For all $\lbb>0$, we have
\begin{equation}\label{e3.11}
(I+\lbb A_\vp)\1u_0
{\buildrel{\vp\to0}\over\longrightarrow}
(I+\lbb A)\1u_0\mbox{ strongly in $L^2$}.  \end{equation}
\end{lemma}

\begin{proof} We set $v_\vp=(I+\lbb A_\vp)\1u_0$, that is, $v_\vp\in H^2$, and 
\begin{equation}\label{e3.12}
v_\vp-\lbb\vp\Delta v_\vp-\lbb\divv(|\nabla v_\vp|^{p-2}\nabla v_\vp)=u_0\mbox{ in }\cald'(\rrd).
\end{equation}\newpage
Since $L^1\cap L^\9$ is dense in $L^2$ and $(I+\lbb A_\vp)\1$, $\vp>0$, are contractions (hence equicontinuous) on $L^2$, it suffices to prove \eqref{e3.12} for $u_0\in L^2\cap L^\9$. But then, as in the proof of Theorem \ref{t2.4}, it follows that
\begin{equation}\label{e3.13b}
	|v_\vp|_\9\le|u_0|_\9,\ \ff\vp>0.
\end{equation}
Hence, by \eqref{e3.13b} and an approximation argument, we have the estimate
\begin{equation}\label{e3.13a}
\frac12|v_\vp|^2_2+\lbb\vp|\nabla v_\vp|^2_2+\lbb|\nabla v_\vp|^p_p\le\frac12|u_0|^2_2,\ \ff\vp>0,\end{equation}
and
$$\barr{ll}
\dd\frac1p\,|v_\vp|^p_p\!\!\!
&+\lbb\vp(p-1)\dd\int_\rrd|\nabla v_\vp|^2|v_\vp|^{p-2}dx\vsp&+\lbb(p-1)\dd\int_\rrd|\nabla v_\vp|^p|v_\vp|^{p-2}dx
\le\dd\frac1p\,|u_0|_p.\earr$$
Hence,
\begin{equation}\label{e3.13aa}
|v_\vp|_p\le|u_0|_p,\ \ff u_0\in L^2\cap L^p,\end{equation}
and this implies that along a subsequence, again denoted $\{\vp\}\to0$, we have
\begin{equation}\label{e3.13}
\barr{rcll}
v_\vp&\to&v&\mbox{strongly in $L^2\loc$ and weakly in $L^2$}\vsp 
\nabla v_\vp&\to&\nabla v&\mbox{weakly in $(L^p)^d$}\vsp
\vp\nabla v_\vp&\to&0&\mbox{strongly in $(L^2)^d$}\vsp
|\nabla v_\vp|^{p-2}\nabla v_\vp&\to&\eta&\mbox{weakly in }(L^{p'})^d.\earr\end{equation}
By \eqref{e3.12}--\eqref{e3.13} it follows that
\begin{equation}\label{e3.14}
v-\lbb\,\divv\,\eta=u_0\mbox{ in }\cald'(\rrd).\end{equation}
On the other hand, we have
$$\barr{l}
-\dd\int_\rrd\divv(|\nabla v_\vp|^{p-2}\nabla v_\vp)(v_\vp-w)dx\ge\frac1p\int_\rrd|\nabla v_\vp|^pdx\vsp
\hfill-\dd\frac1p\int_\rrd|\nabla w|^p dx,\ \ff w\in L^2,\ |\nabla w|\in L^p.\earr$$
This yields
$$\frac\lbb p\int_\rrd|\nabla v_\vp|^pdx-\int_\rrd(u_0-v_\vp+\vp\lbb\Delta v_\vp)(v_\vp-w)dx\le\frac\lbb p\int_\rrd|\nabla w|^pdx$$
and for $\vp\to0$ it follows by \eqref{e3.13} that\newpage
$$\dd\frac\lbb p\int_\rrd|\nabla v|^pdx\le\frac\lbb p\int_\rrd|\nabla w|^pdx
+\dd\int_\rrd(u_0-v)(v-w)dx,\ \ff w\in L^2,\ |\nabla w|\in L^p,$$
which means that $\divv\,\eta\in\pp\Phi(v)$, where $\Phi$ is the function \eqref{e2.4}. Hence, by Proposition \ref{p2.1},  $\divv\,\eta=-\divv|\nabla v|^{p-2}\nabla v,$ a.e. in $\rrd$, and thus $v\in D(A)$ and $v=(I+\lbb A)\1u_0$.

To conclude the proof, it remains to be shown that 
\begin{equation}\label{e3.15}
v_\vp\to v\mbox{ strongly in $L^2$ as $\vp\to0$}.\end{equation}
By the density of $L^p\cap L^2$ in $L^2$ and since $|(I+\lbb A_\vp)\1u-(I+\lbb A_\vp)\1v|_2\le|u-v|_2$, $\ff u,v\in L^2$, it suffices to prove \eqref{e3.15} for $u_0\in L^2\cap L^p$.

Let $\eta\in C^2([0,\9))$ be such that $\eta\ge0$, $\eta(r)=0$, $\ff r\in[0,1]$, $\eta(r)=1$, $\ff r\ge2$. If we multiply \eqref{e3.12} by $\eta_n v_\vp$, where $\eta_n(x)=\eta\(\frac{|x|}n\)$, we get
$$\barr{l}
\dd\int_\rrd v^2_\vp(x)\eta_n(x)dx+\lbb\vp\int_\rrd\nabla v_\vp(x)\cdot
\(\eta_n(x)\nabla v_\vp(x)+\frac1n\eta'\(\frac{|x|}n\)\frac x{|x|} v_\vp(x)\)dx\vsp 
\qquad+\lbb\dd\int_\rrd|\nabla v_\vp(x)|^{p-2}\nabla v_\vp(x)\cdot
\(\eta_n(x)\nabla v_\vp(x)+\frac1n\eta'\(\frac{|x|}n\)\frac x{|x|}v_\vp(x)\)dx\vsp 
\qquad=\dd\int_\rrd u_0(x)v_\vp(x)\eta_n(x)dx.
\earr$$
This yields
$$\barr{l}
\dd\frac12\int_\rrd v^2_\vp(x)\eta_n(x)dx
+\dd\frac{\lbb\vp}n
\int_\rrd\(\nabla v_\vp(x)\cdot\frac x{|x|}\)\eta'\(\frac{|x|}n\)v_\vp(x)dx\vsp 
\qquad+\dd\frac\lbb n\int_\rrd|\nabla v_\vp(x)|^{p-2}
\(\nabla v_\vp(x)\cdot\frac x{|x|}\)\eta'\(\frac{|x|}n\)v_\vp(x)dx\vsp 
\qquad\le\dd\frac12\int_\rrd u^2_0(x)\eta_n(x)dx.\earr$$
Taking into account that by \eqref{e3.13a}--\eqref{e3.13aa}\vspace*{-3mm}
$$\barr{r}
\dd\int_\rrd(|\nabla v_\vp|\,|v_\vp|dx
+|\nabla v_\vp|^{p-1}|v_\vp|)dx\le|\nabla v_\vp|_2|v_\vp|_2+|\nabla v_\vp|^{p-1}_p|v_\vp|_p\le C,\\ \ff \vp>0,\earr$$\vspace*{-6mm}
we get
$$\int_{[|x|\ge2n]}v^2_\vp(x)dx\le\frac Cn+\int_{[|x|\ge n]}u^2_0(x)dx,\ \ff \vp>(0,1),\ n\in\nn,$$
where $C$ is independent of $\vp$.  
Combined with \eqref{e3.13}, the latter yields \eqref{e3.15}.~\end{proof}

\n{\it Proof of \eqref{e3.2a} $($continued$)$.} For each $\vp>0$ and $i=1,...,d$ we set
$$w^\vp_i=D_iu_\vp,\ w^\vp=(w^\vp_i)^d_{i=1}=\nabla u_\vp.$$
Taking into account that by \eqref{e3.9}, $u_\vp\in L^2(0,T;H^2)$, $\ff T>0$, we have $w^\vp_i\in L^2(0,T;H^1)$. Differentiating formally \eqref{e3.3} with respect to each $i$, we~get\vspace*{-3mm}
\begin{equation}\label{e3.16}
	\barr{ll}
	\dd\frac{\pp w^\vp_i}{\pp t}=\divv(|w^\vp|^{p-2}\nabla w^\vp_i
	+(p-2)|w^\vp|^{p-4}w^\vp(\nabla w^\vp_i{\cdot}w^\vp))\\
	\hfill	
	+\vp\Delta w^\vp_i\ \mbox{in }\cald'((0,\9)\times\rrd),\vsp
	w^\vp_i(0,x)=D_iu_0(x),\ \ i=1,2,...,d.\earr\end{equation}
One might suspect that $w^\vp_i$ is the strong solution to \eqref{e3.16} in the space $L^2(0,\9;H^1)$. 
To prove rigorously the exact significance of \eqref{e3.16} we consider the finite differences
$$\barr{r}
u^h_{\vp,i}(t,x)=\dd\frac1h(u_\vp(t,x+he_i)-u_\vp(t,x)),\ e_i=(0,0,...,1,0,...,0),\vsp
(t,x)\in(0,\9)\times\rrd,\earr$$
and get 
\begin{equation}
\barr{l}
	\dd\frac\pp{\pp t}u^h_{\vp,i}(t,x)=\divv
	(|\nabla u_\vp(t,x+h e_i)|^{p-2}
	\nabla u_\vp(t,x+he_i)\label{e3.19a}\vsp
	 \quad\qquad-|\nabla u_\vp(t,x)|^{p-2}\nabla u_\vp(t,x))	+\vp\Delta u^h_{\vp,i}(t,x) 
 \mbox{ in }\cald'((0,\9)\times\rrd,\vsp
 	u^h_{\vp,i}(0,x)=\dd\frac1h(u_0(x+he_i)-u_0(x)),\ x\in\rrd.\earr\end{equation}
Define the function $F:\rrd\to\rrd$ by
$$F(x)=|x|^{p-2}x, \ \ff x\in\rrd.$$
We have
$$\barr{ll}
F(y)-F(x)\!\!\!&=\dd\int^1_0\frac d{d\tau}F(\tau y-(1-\tau)x)d\tau\vsp
&=\dd\int^1_0 DF(\tau y-(1-\tau)x)\cdot(y-x)d\tau,\  \ff x,y\in\rrd,\earr$$
where $DF=(D_jF_i)_{i,j=1}$ and
$$D_jF_i(x)=(p-2)|x|^{p-4}x_ix_j+\delta_{ij}|x|^{p-2},\ i,j=1,...,d.$$
We set
$$f^h_{\vp,i}(t,x,\tau)=\nabla(\tau u_\vp(t,x+he_i)+(1-\tau)u_\vp(t,x)),\ \tau\in(0,1),\ (t,x)\in(0,\9)\times\rrd.$$
Then, \eqref{e3.19a} turns into
\begin{eqnarray}
&&\hspace*{-8mm}	\dd\frac\pp{\pp t}u^h_{\vp,i}(t,x)\label{e3.21b}\\ 
&&\hspace*{-8mm}=\divv\((p-2)\dd\int^1_0|f^h_{\vp,i}(t,x,\tau)|^{p-4}
	f^h_{\vp,i}(t,x,\tau)(f^h_{\vp,i}(t,x,h)\cdot\nabla u^h_{\vp,i}(t,x))d\tau\right.\nonumber\\
&&\hspace*{-8mm}	\left.+\dd\int^1_0|f^h_{\vp,i}(t,x,\tau)|^{p-2}d\tau
	\,\nabla u^h_{\vp,i}(t,x)\)+\vp\Delta u^h_{\vp,i}(t,x).
	\nonumber\end{eqnarray}
Taking into account that $u^h_{\vp,i}\in L^2(0,\9;H^2)$, we get by \eqref{e3.21b}
\begin{eqnarray}
&&	\dd\frac12
	|u^h_{\vp,i}(t)|^2_2+ 
	\dd\int^t_0ds\int_\rrd
	\(\int^1_0|f^h_{\vp,i}(s,x,\tau)|^{p-2}d\tau\)
	|\nabla u^h_{\vp,i}(s,x)|^2dx\label{e3.22b}\\[1mm] 
	&& +(p-2)\dd\int^t_0ds\int_\rrd
	\(\dd\int^1_0|f^h_{\vp,i}(s,x,\tau)|^{p-4}
	(f^h_{\vp,i}(s,x,h)\cdot
	\nabla u^h_{\vp,i}(s,x))^2 dxd\tau\)\nonumber\\[1mm]
	&&  +\vp\dd\int^t_0\!\!ds\int_\rrd
	|\nabla u^h_{\vp,i}(s,x)|^2dx=\dd\frac12|u^h_{\vp,i}(0)|^2_2,\ \ff t\ge0. 
	\nonumber\end{eqnarray}
On the other hand, we know that for $h\to0$, 
$$\barr{rcll}
\nabla u^h_{\vp,i}&\to&\nabla w^\vp_i&\mbox{ in }L^2(0,T;L^2),\vsp 
|\nabla u_\vp(t,x+he_i)|^{p-2}&\to&|\nabla u_\vp(t,x)|^{p-2}&\mbox{ in }L^{\frac p{p-2}}((0,T)\times\rrd).\earr$$
Hence, for $h\to0$, along a subsequence, again denoted $\{h\}\to0$,
$$\nabla u^h_{\vp,i}(t,x)\to \nabla w^\vp_i(t,x),\ |\nabla u_{\vp,i}(t,x+h e_i)|\to|\nabla u_\vp(t,x)|,\mbox{ a.e. }(t,x)\in(0,T)\times\rrd,$$
and so, by \eqref{e3.22b} and by   Fatou's lemma, we get  the estimate
\begin{equation}\label{e3.19aaaa}
	\barr{ll}
	\dd\frac12|w^\vp_i(t)|^2_2+\int^t_0ds\int_\rrd
	(|w^\vp(s,x)|^{p-2}|\nabla w^\vp_i(s,x)|^2\vsp
	\qquad+(p-2)|w^\vp(s,x)|^{p-4}
	(w^\vp(s,x){\cdot}\nabla w^\vp_i(s,x))^2)dx\vsp
	\qquad+\vp\dd\int^t_0ds\int_\rrd|\nabla w^\vp_i(s,x)|^2dx\le\frac12|D_iu_0|^2_2,\ i=1,2,...,d.\earr
\end{equation}
Moreover, we have 
\begin{eqnarray}
w^\vp_i\in L^\9(0,\9;L^1\cap L^\9)),&&\hspace*{-3mm} \ff i=1,...,d,\label{e3.18}\\[1mm]
|w^\vp_i(t)|_\9\le|w^\vp_i(0)|_\9=|D_iu_0|_\9,&&\hspace*{-3mm}\ \ff t\ge0,\ i=1,...,d,\label{e3.19}\\[1mm]
|w^\vp_i(t)|_1\le|w^\vp_i(0)|_1=|D_iu_0|_1,&&\hspace*{-3mm}\ \ff t\ge0,\ i=1,...,d.\label{e3.20}
\end{eqnarray}
Indeed, if $M_i=|u^h_{\vp,i}(0)|_\9$, we have by \eqref{e3.21b}  
$$\barr{l}
\dd\frac\pp{\pp t}(u^h_{\vp,i}-M_i)\vsp
= \divv\!\(\!(p{-}2)\dd\int^1_0\!\!|f^h_{\vp,i}(t,x,\tau)|^{p-4}
 f^h_{\vp,i}(t,x,\tau)
 (f^h_{\vp,i}(t,x,h){\cdot}\nabla(u^h_{\vp,i}(t,x){-}M_i))d\tau\right.\vsp 
 \left. +\dd\int^1_0|f^h_{\vp,i}(t,x,\tau)|^{p-2}d\tau
 \nabla(u^h_{\vp,i}(t,x)-M_i)\)+
 \vp\Delta(u^h_{\vp,i}(t,x)-M_i).
 \earr$$
 Multiplying by $(u^h_{\vp,i}-M_i)^+$ and integrating over $\rrd$ yields
$$\frac d{dt}|(u^h_{\vp,i}(t)-M_i)^+|^2_2\le0,\mbox{ a.e. }t>0,$$and, therefore, $u^h_{\vp,i}(t,x)\le M_i$, a.e. $(t,x)\in(0,\9)\times\rrd.$ Similarly, it follows that $u^h_{\vp,i}(t,x)\ge-M_i$, a.e. $(t,x)\in(0,\9)\times\rrd.$
 Hence, \eqref{e3.19} follows.
 
 In particular, \eqref{e3.19aaaa} 
 implies that
 $$|w^\vp|^{\frac{p-2}2}|\nabla w^\vp_i|,\
 |w^\vp|^{\frac{p-4}2}(w^\vp\cdot\nabla w^\vp_i)\in L^2(0,\9;L^2),\ i=1,...,d,$$
 and so, by \eqref{e3.19a} and \eqref{e3.19aaaa} it follows that $w^\vp_i$ is the solution to \eqref{e3.16} in the following strong sense
 \begin{equation}
 w^\vp_i\in L^2(0,T;H^1),\ \dd\frac{dw^\vp_i}{dt}\in L^2(0,T;H\1), \ T>0,\label{e3.27b}\end{equation}
 \begin{eqnarray}
&&\hspace*{-7mm}	\dd\frac{dw^\vp_i}{dt}(t)=\divv(|w^\vp(t)|^{p-2}\nabla w^\vp_i(t)
	+(p-2)|w^\vp(t)|^{p-4}w^\vp(t)
	(\nabla w^\vp(t){\cdot}\nabla w^\vp_i(t)))\nonumber\\ 
&&	\ \ \qquad+\vp\Delta w^\vp_i(t),\mbox{ a.e. }t\in(0,\9),\mbox{ in }H\1,\label{e3.23a}\\
&&\hspace*{-7mm}	w^\vp_i(0)=D_iu_0.\nonumber
	\end{eqnarray}
 Next, to get \eqref{e3.20} we multiply \eqref{e3.23a} by $\calx_\delta(w^\vp_i)$, where $\calx_\delta$ is the function \eqref{e2.22a}, and let $\delta\to0$. We omit the details (see the proof of \eqref{e2.22}).

In particular, it follows by \eqref{e3.19aaaa} that
\begin{equation}\label{e3.28b}\barr{r}
\dd\int^\9_0\!\!\!\!dt\int_\rrd
\( |w^\vp(t,x)|^{p-2}|\nabla w^\vp_i(t,x)|^2
+\vp|\nabla w^\vp_i(t,x)|^2\)dx\le\frac12|D_iu_0|^2_2,\\ \ff i=1,...,d.\earr\end{equation}
This yields
\begin{equation}\label{e3.22}
\int^\9_0dt\int_\rrd|\nabla(|w^\vp(t,x)|^{\frac p2})|^2 dxdt
\le\frac{dp^2}8|\nabla u_0|^2_2.\end{equation}
Since by \eqref{e3.18} $\{w^\vp_i\}_\vp$ is bounded in $L^2((0,T)\times \rrd)$, $\ff T>0$, it follows in particular  that along a subsequence, again denoted $\{\vp\}\to0$,
\begin{eqnarray}
	w^\vp_i&\to&w_i\ \ \mbox{weakly  in }L^2((0,T)\times\rrd),\ \ff T>0,\label{e3.23}
	\end{eqnarray}
where by \eqref{e3.10} and the closedness of the operator $D_i$ on $L^2(0,T;L^2)$, \mbox{$w_i=D_iu$.}

By \eqref{e3.19} it follows also that
\begin{equation}\label{e3.27a}
	|\nabla u|\in L^\9((0,T)\times L^\9).\end{equation}
We denote by 
$\frac d{dt}(|w^\vp_i(t)|^{\frac p2-1}w^\vp_i(t))$  
the distributional derivative of the function $t\to|w^\vp_i(t)|^{\frac p2-1}w^\vp_i(t)$, that is,
$$\barr{r}
\raise-2,9mm\hbox{${{}}_{\cald'}$}\!\!\<\dd\frac d{dt}|w^\vp_i(t)|^{\frac p2-1}w^\vp_i(t),\vf\>_{\!\!\cald}
=-\dd\int^T_0\int_\rrd
\frac d{dt}|w^\vp_i(t)|^{\frac p2-1}w^\vp_i(t)
\frac{d\vf}{dt}(t,x)dxdt,\vsp
\ff\vf\in C^\9_0((0,T)\times\rrd)=:\cald.\earr$$

\begin{claim}\label{claim} Let $1<q<\min\(\frac d{d-1},2\)$ if $d\ge3,$ $q\in(1,2)$ if $d=2$, and $q=2$ if $d=1$. Then, 
	\begin{equation}\label{e3.32}
		\sup_{0<\vp\le1}\int^T_0\left\|\frac d{dt}(|w^\vp_i(t)|^{\frac p2-1}w^\vp_i(t))\right\|_{W^{-1,q}}dt<\9.
	\end{equation}
\end{claim}

\begin{proof} 
	To prove \eqref{e3.32},   it suffices to mention that by \eqref{e3.27b}  the function $t\to w^\vp_i(t)$ is $H\1$-valued absolutely continuous on $[0,T]$ and a.e. $t\in(0,T)$,
	$$\frac d{dt}w^\vp_i(t)=\lim_{h\to0}\ \frac{w^\vp_i(t+h)-w^\vp_i(t)}h\mbox{\ \ strongly in }H\1,\mbox{ hence in }W^{-1,q}.$$
	This yields

\begin{equation}\label{e3.33a}
\hspace*{-4mm}\barr{l}	\dd\int^T_0\!\!\int_\rrd\!|w^\vp_i(t)|^{\frac p2-1}w^\vp_i(t)\frac{d\vf}{dt}dxdt\vsp
 =\dd\lim_{h\to0} \frac1{h}\int^{T-h}_0\!\!\int_\rrd
|w^\vp_i(t)|^{\frac p2-1}w^\vp_i(t)
(\vf(t+h,x)-\vf(t,x))dxdt\vsp
=\dd\lim_{h\to0} \frac1{h}\int^{T}_h\!\!\!\int_\rrd
\!(|w^\vp_i(t{-}h,x)|^{\frac p2-1}w^\vp_i(t{-}h,x)
{-}|w^\vp_i(t,x)|^{\frac p2-1}w^\vp_i(t,x))
\vf(t,x)dxdt\vsp
=-\dd\lim_{h\to0}\int^T_hdt\int_\rrd\int^1_0
\frac p2|\tau w^\vp_i(t,x)+(1-\tau)w^\vp_i(t-h,x)|^{\frac{p-2}2}d\tau\vsp
 \ \qquad\quad\qquad\qquad\quad\qquad\ \ \ 
\dd\frac1h(w^\vp_i(t,x)-w^\vp_i(t-h,x))\vf(t,x)dx\vsp
=-\dd\lim_{h\to0}\frac p2\int^T_h\!\!\!dt
\raise-2,9mm\hbox{${}_{H\1}$}\!\!\!\<\!\frac1h (w^\vp_i(t){-}w^\vp_i(t{-}h)),\!
\dd\int^1_0\!\!\!
|\tau w^\vp_i(t){+}(1{-}\tau)w^\vp_i(t{-}h)|^{\frac{p-2}2}
d\tau\,\vf(t)\!\>_{\!\!H^1}\vsp
=-\dd\frac p2\int^T_0
\raise-2,9mm\hbox{${}_{H\1}$}\!\!\<\frac{dw^\vp_i(t)}{dt}
|w^\vp_i(t)|^{\frac{p-2}2},\vf(t)\>_{\!H^1}dt,\ 
\ff\vf\in C^\9_0((0,T)\times\rrd),\earr\hspace*{-16mm}\end{equation}
 where we used that the functions in the second slot of the dualization above are bounded in $L^2(0,T;H^1)$, uniformly in $h$, hence (selecting a subsequence, if necessary) weakly converge in $L^2(0,T;H^1)$ to
 $|w^\vp_i|^{\frac{p-2}2}\,\vf$, as $h\to0$. We note that we  have also used above that \eqref{e3.27b}   implies $w_\vp\in C([0,T];L^2)$. So, to prove \eqref{e3.32}, it suffices to show that the last expression in \eqref{e3.33a} is a continuous (linear) function in $\vf$ with respect to the norm of $L^2(0,T;W^{1,\frac q{q-1}})$, whose norm is uniformly bounded in $\vp\in(0,1]$. To this end, we first note that by \eqref{e3.23a} the last expression in \eqref{e3.33a}   is equal to
\begin{equation}\label{e3.34prim}
\frac p2\int^T_0\int_\rrd \(f(t,x)\cdot(\vf(t,x)\nabla g(t,x)+g(t,x)\nabla\vf(t,x))\)dx\,dt,\end{equation}
where 
$$\barr{l}
f=|w^\vp|^{p-2}\nabla w^\vp_i+(p-2)w^\vp(w^\vp\cdot\nabla w^\vp_i)|w^\vp|^{p-4}+\vp\nabla w^\vp_i,\vsp
g=|w^\vp_i|^{\frac{p-2}2}.\earr$$
But, since $\frac q{q-1}>d$, the absolute value of the term in \eqref{e3.34prim} up to a constant is dominated by

$$\barr{l} \dd\int^T_0
(|f(t)\nabla g(t)|_1\|\vf(t)\|_{W^{1,\frac q{q-1}}}
+|f(t)|_2|g(t)|_{\frac{2q}{2-q}}|\nabla\vf(t)|_{\frac q{q-1}})dt\vsp 
\qquad\le\dd\int^T_0(|f(t)\nabla g(t)|_1+|f(t)|^2_2+|g(t)|^2_{\frac{2q}{q-2}})dt 
\|\vf\|_{L^\9(0,T;W^{1,\frac q{q-1}})}\vsp
\qquad\le C_p\(|D_iu_0|^{\frac{p-4}2}_\9
|D_iu_0|^2_2\right.\vsp
\qquad\left.+(|\nabla u_0|^{p-2}_\9+1)|D_iu_0|^2_2
+T|D_iu_0|^{\frac{2(2-q)}q}_2|D_iu_0|^{\frac{4(p-1)}q}_\9\)
\|\vf\|_{L^\9(0,T;W^{1,\frac q{q-1}})},\earr$$
where we have used \eqref{e3.19aaaa} and \eqref{e3.19}. 
So, \eqref{e3.32} is proved. 

We note that
$$|\nabla(|w^\vp_i|^{\frac p2})|=
|\nabla(w^\vp_i|^{\frac{p-2}2}w^\vp_i)|,\mbox{ a.e. on }(0,T)\times\rrd,$$
and so, by \eqref{e3.22} and \eqref{e3.32} we have
$$\int^T_0\int_\rrd|\nabla(\eta^\vp_i(t,x))|^2dx\,dt
+\int^T_0\left\|\frac d{dt}(\eta^\vp_i(t))\right\|_{W^{-1,q}}dt\le C,$$
where
$$\eta^\vp_i(t,x)=|w^\vp_i(t,x)|^{\frac{p-2}2}w^\vp_i(t,x),\mbox{ a.e. }(t,x)\in(0,T)\times\rrd,$$
and $C$ is independet of $\vp$. In particular, it follows that
$$\int^T_0\(\int_{B_R}|\nabla\eta^\vp_i(t,x)|^2dx+\left\|\frac d{dt}\eta_i(t)\right\|_{W^{-1,q}(B_R)}\)dt\le C,$$
for each ball $B_R=\{x;\,|x|<R\}$.

Then, by the Aubin-Lions--Simon compactness theorem (\cite{8a}) applied to the spaces $H^1(B_R)\subset L^2(B_R)\subset W^{-1,q}(B_R)$, where $R>0$ is arbitrary,    we infer that, for each $i=1,...,d$, the sequence $\{\eta^{\vp_n}_i\}_n$ is compact  
in $L^2(0,T;L^2_{\rm loc})$. Hence, along a subsequence \mbox{$\{\vp_n\}\to0$,} as $n\to\9$, we have
$$\eta^{\vp_n}_i\to\eta_i\mbox{ strongly in }L^2(0,T;L^2_{\rm loc})$$
and so, selecting the further subsequence,
\begin{equation}\label{e3.37d}
	\eta^{\vp_n}_i(t,x)\to\eta_i(t,x),\mbox{ a.e. }(t,x)\times\rrd.\end{equation}
Taking into account that the function $z\to|z|^{\frac{p-2}2}\,z$ is monotonically increasing on $\rr$, we have
$$w^{\vp_n}_i\to\wt w_i,\mbox{ a.e. in }(0,T)\times\rrd$$
and, recalling \eqref{e3.23}, we infer that $\wt w_i=w_i$ and so, by \eqref{e3.37d} it follows that $\eta_i=|w_i|^{\frac{p-2}2}w_i$, a.e. on $(0,T)\times\rrd$ and, therefore, it follows that, for $n\to\9$,
$$|w^{\vp_n}_i|^{\frac p2}\to|w_i|^{\frac p2}\mbox{ strongly in $L^2(0,T;L^2_{\rm loc})$},\ i=1,...,d.$$
Since, by \eqref{e3.23}, $\{|\nabla(w^{\vp_n})|\}_n$ is bounded in $L^2(0,T;L^2)$, we also have
$$\nabla(|w^{\vp_n}|^{\frac p2})\to\nabla(|w|^{\frac p2})\mbox{ weakly in }L^2(0,T;(L^2)^d),$$
and so, by weak lower-semicontinuity it follows by \eqref{e3.22} that
$$\int^\9_0\int_\rrd|\nabla(|w|^{\frac p2})|^2dxdt\le
\frac{dp^2}8|\nabla u_0|^2_2.$$	
Hence, we have	
\begin{equation}\label{e3.38d}
\int^T_0\int_\rrd
|\nabla|\nabla u|^{\frac p2}|^2dxdt
\le\frac{dp^2}8 |\nabla u_0|^2_2.\end{equation}
This yields
\begin{equation}
\label{e3.32a}
\barr{ll}
\dd\int^T_0\!\! \int_\rrd
|\nabla(|\nabla u|^{p-2})|dxdt   
=\frac{2(p-2)}p\dd\int^T_0\int_\rrd
|\nabla(|\nabla u|^{\frac p2})|\,|\nabla u|^{\frac{p-4}2}dxdt\vsp
\qquad\qquad\le\sqrt{\dd\frac d2}\,(p-2)|\nabla u_0|_2
\(\dd\int^T_0|\nabla u(t)|^{{p-4}}_{p-4}dt\)^{\frac12}\vsp 
\qquad\qquad\le\sqrt{\dd\frac d2}\,(p-2)|\nabla u_0|_2
(V\!ol(B_{R(T)}))^{\frac 2p}
\(\dd\int^T_0|\nabla u(t)|^{p}_pdt\)^{\frac{p-4}{2p}}\vsp 
\qquad\qquad\le2^{\frac2p}\,\sqrt{d}\,
\dd\frac{p-2}{2}|\nabla u_0|_2
|u_0|^{\frac{p-4}p}_2\ T^{\frac2p}
(V\!ol(B_{R(T)}))^{\frac 2p},\ \ff T>0,\earr\hspace*{-10mm}\end{equation}
where we used \eqref{e2.18a}.\end{proof}

\begin{remark}\label{r3.3}\rm There are several earlier works concerning the first order re\-gu\-la\-rity for weak solutions to the parabolic $p$-Laplace equation \eqref{e1.1} (see, e.g., \cite{1c}, \cite{5prim}, \cite{8d}, \cite{6}). However, it  seems that the estimate \eqref{e3.3} is new. As mentioned earlier, a related local second-order result was established in \cite{9b} (see also \cite{12prim}). But, therein, integrability with respect to time is only proved over time intervals $(\delta,T]$ with $\delta>0$. We, however, crucially need integrability over $[0,T]$, as will be seen in the next section.  
\end{remark}

\section{The probabilistic representation}\label{s4}
\setcounter{equation}{0}

We shall assume here, as in Section \ref{s4}, that 
$p\ge4$ and
\begin{eqnarray}
	&u_0\in W^{1,\9},\ \ u_0\ge0,\ \mbox{a.e. in }\rrd\label{e4.1}\\[1mm]
	&support\ u_0\subset B_R,\ \dd\int_\rrd u_0(x)dx=1.\label{e4.3}
	\end{eqnarray}
Then, as seen earlier in Theorem \ref{t2.3} and Theorem \ref{t2.4}, equation \eqref{e1.1} has a unique weak solution $u$ on $(0,\9)\times\rrd$,  satisfying \eqref{e2.17}--\eqref{e2.20} and \begin{eqnarray}
	&u\in C([0,\9);L^2)\cap L^\9(0,\9;L^2),\ u\ge0,
	\mbox{ a.e. on }(0,\9)\times\rr,\label{e4.4}\\[1mm]
	&u\in L^2(0,\9;H^1)\cap L^\9((0,\9)\times\rrd),\label{e4.5}\\[1mm]
	&\dd\int_\rrd u(t,x)dx=1,\ \ff t\ge0.\label{e4.6}
	\end{eqnarray}
Moreover, the estimates \eqref{e3.2}--\eqref{e3.2a} hold. 

As seen earlier in \eqref{e1.2},   we may rewrite equation \eqref{e1.1} as the Fokker--Planck equation 
\begin{equation}\label{e4.6a}
	\barr{l}
	\dd\frac{\pp u}{\pp t}-\Delta(|\nabla u|^{p-2}u)+\divv(\nabla(|\nabla u|^{p-2})u)=0\mbox{ in }\cald'((0,\9)\times\rrd),\vsp
	u(0,x)=u_0(x),\ x\in\rrd,\earr
\end{equation}
and taking into account that $u\in L^\9((0,T)\times\rrd)$, by \eqref{e3.2}, \eqref{e3.2a},  we have the estimates
\begin{eqnarray}
\int^T_0\!\!\!\int_\rrd|\nabla u(t,x)|^{p-2}|u(t,x)|dxdt
&\!\!\!\le\!\!\!&
C_T(R)|u_0|^{\frac{2(p-2)}2}_2,\label{e4.7}\\[2mm]
\int^T_0\!\!\!\int_\rrd
\!|\nabla(|\nabla u(t,x)|^{p-2})|
|u(t,x)|dxdt
&\!\!\!\le\!\!\!&
C^1_T(R)T|u_0|^{\frac{p-4}p}_2
|\nabla u_0|_2,\   \qquad\label{e4.8}
\end{eqnarray}for all $T>0$. Moreover, by Proposition \ref{p2.5}, we have that
\begin{equation}\label{e4.9}
support\ u(t)\subset B_{2R+R(t)},\ \ff t\ge0.
\end{equation}
Also, by \eqref{e2.16}  it follows that $u:[0,\9)\to L^2$ is continuous. Then, we finally obtain:

\begin{theorem}\label{t4.1} Let $u$ be the solution to \eqref{e1.1} under hypotheses \eqref{e4.1}--\eqref{e4.3} on the initial data $u_0$. Then, there exists a $d$-dimensional $(\calf_t)$-Brownian motion on a stochastic basis $(\ooo,\calf,(\calf_t)_{t\ge0},\mathbb{P})$ and an $(\calf_t)$-progressively measurable map $X:[0,\9)\times\ooo\to\rrd$, continuous in $t$, solving  the McKean--Vlasov stochastic differential equation
\begin{equation}\label{e4.10}
	dX(t)=\nabla(|\nabla u(t,X(t))|^{p-2})dt
	+\sqrt{2}|\nabla u(t,X(t))|^{\frac{p-2}2}dW(t)
\end{equation}
such that
\begin{equation}\label{e4.11}
	u(t,x)dx=\mathbb{P}\circ X(t)\1(dx),\ t\ge0.
	\end{equation}
\end{theorem}	

\begin{proof} This is a direct consequence of \cite[Section 2]{1a} (see also \cite[Chapter 5]{2} and the proof of \cite[Theorem 3.3]{3}).\end{proof}



\begin{thebibliography}{nn}
	
 \bibitem{1c} Alikakos, N.D., Evans, L.C., Continuity of the gradient for weak solutions to degenerate parabolic equations, {\it J. Math. Pures Appl.}, 62 (3) (1983), 253-268.\vspace*{-1,3mm}  
 
	
\bibitem{1} Barbu, V.,  {\it Nonlinear Differential Equations of Monotone Type in Banach Spaces}, Springer,  New York,  2010.\vspace*{-1,3mm} 

\bibitem{1a} Barbu, V., R\"ockner, M.,  From nonlinear Fokker--Planck equations to solutions of distribution dependent SDEs, {\it Annals Probab.}, 48 (4) (2020), 1902-1920.\vspace*{-1,3mm}    

\bibitem{2} Barbu, V., R\"ockner, M.,   Nonlinear  Fokker--Planck flows and their pro\-ba\-bi\-listic counterparts, Lecture Notes in Mathematics, 2353 (2024), Springer.\vspace*{-1,3mm}

\bibitem{3} Barbu, V., 
Rehmeier, M., R\"ockner, M.,   $p$-Brownian motion and the $p$-Laplacian, {\it Annals of Probability} (to appear).\vspace*{-1,3mm} 

	
\bibitem{4} Barbu, V., Gruber, G., Rehmeier, M., R\"ockner, M., The Leibenson Process, arXiv: 2508.12979.\vspace*{-1,3mm}

\bibitem{5c} Beznea, L., C\^\i mpean, I., R\"ockner, M., Regularization of the superposition principle: Potential theory meets Fooker--Planck equations, (2026) arXiv:2603.05174v.2.  	
https://doi.org/10.48550/arXiv.2603.05174\vspace*{-1,3mm} 


\bibitem{5prim} B\"ogelein, V., Ragnedda, F., Vernier Piro, S., Vespri, V., Moscr--Nash kernel estimates for degenerate parabolic equations, {\it J. Funct. Anal.}, 272 (2017), 2956-2986.\vspace*{-1,3mm}


\bibitem{8b} Brezis, H., {\it Functional Analysis, Sobolev Spaces and Partial Differential Equations}, Springer, New York. Dordrecht. Heidelberg. London, 2010.\vspace*{-1,3mm}


\bibitem{8d} Di Benedetto, E., Friedman, A., Regularity of solutions to nonlinear degenerate parabolic systems, {\it J. reine Angew. Math.}, 349 (1984), 83-128. \vspace*{-1,3mm}

\bibitem{6} Duzgun, F.G., Mosconi, S., Vespri, V., Anisotropic Sobolev embedding and the speed of propagation for parabolic equations, {\it J. Evol. Equations}, 19 (2019), 845-882.\vspace*{-1,3mm} 

\bibitem{9b} Feng, Y., Parviainen, M., Sarsa, S., On the second-order regularity of solutions to the parabolic $p$-Laplace equation, {\it J. Evol. Equ.}, 22 (2022), article nr. 6.\vspace*{-1,3mm}

\bibitem{12prim} Feng, Y., Parviainen, M., Sarsa, S., Second-order Sobolev regularity results for the generalized $p$-parabolic equation, {\it J. Funct. Anal.}, 288 (2025), 9-27.\vspace*{-1,3mm}


\bibitem{8a} Simon, J., Compact sets in $L^p(0,T;B)$, {\it Ann.Mat. Pura Appl.}, 146 (1987), 65-96.\vspace*{-1,3mm}

\bibitem{9bb} McKean, Jr. H.P., A class of Markov processes associated with nonlinear parabolic equations, {\it Proceedings of the National Academy of Sciences of the United states of America}, 56 (6) (1966), 1907-1911.\vspace*{-1,3mm}

 
\bibitem{7} Trevisan, D., Well-posedness of multi-dimensional diffusion processes with weakly differentiable coefficients, {\it Electron. J. Probab.}, 21 (22) (2016), 1-41.
  	
	
\end{thebibliography}
\end{document}